\documentclass[12pt,reqno]{amsart}
\usepackage{amssymb,latexsym}
\usepackage{amsmath}
\usepackage{enumerate}

\makeatletter
\@namedef{subjclassname@2010}{
  \textup{2010} Mathematics Subject Classification}
\makeatother

\newtheorem{thm}{Theorem}
\newtheorem{cor}[thm]{Corollary}
\newtheorem{lem}[thm]{Lemma}

\theoremstyle{definition}

\newtheorem{rem}[thm]{Remark}

\newtheorem{observation}[thm]{Observation}

\def\la{\lambda}
\def\leq{\leqslant} \def\geq{\geqslant} \def\al{\alpha}
\def\be{\beta}   
 \def\N{\mathbb N} 
\def\Q{\mathbb Q}  \def\eps{\varepsilon}

\def\beq{\begin{equation}}
\def\enq{\end{equation}}

\begin{document}

\baselineskip=17pt

\title[Unimodularity of zeros of self-inversive polynomials]{ Unimodularity of zeros of self-inversive polynomials}

\author[M. N. Lal\'in]{Matilde N. Lal\'in}
\address{D\'epartement de math\'ematiques et de statistique\\ Universit\'e de Montr\'eal\\CP 6128, succ. Centre-ville\\ Montreal, QC, H3C 3J7, Canada}
\email{mlalin@dms.umontreal.ca}
\thanks{M.L. is supported by NSERC Discovery Grant 355412-2008, FQRNT Subvention \'etablissement
de nouveaux chercheurs 144987, and a start-up grant from the Universit\'e de Montr\'eal.}

\author[C. J. Smyth]{Chris J. Smyth}
\address{School of Mathematics and Maxwell Institute for Mathematical Sciences\\University of Edinburgh\\Edinburgh, EH9 3JZ, Scotland}
\email{c.smyth@ed.ac.uk}

\date{}

\begin{abstract}
We generalise a necessary and  sufficient condition given by Cohn for all the zeros of a self-inversive polynomial to be on the unit circle.
Our theorem implies  some sufficient  conditions found by Lakatos, Losonczi and Schinzel. 
We apply our result to the study of a polynomial family closely related to Ramanujan polynomials, recently introduced by Gun, Murty and Rath, and studied
 by Murty, Smyth and Wang and  Lal\'in and Rogers. We prove that all polynomials
 in this family have their zeros on the unit circle, a result conjectured by Lal\'in and Rogers on computational evidence.
\end{abstract}

\subjclass[2010]{Primary 26C10; Secondary 11B68}

\keywords{reciprocal polynomials, self-inversive polynomials, unit circle, Ramanujan polynomials}

\maketitle

\section{Introduction}

A {\it self-inversive} polynomial of degree $d$ is a nonzero complex polynomial $P(z)$ that satisfies 
\beq\label{E-self}
P(z)=\varepsilon z^d\overline{P}(1/z)
\enq
 for some constant $\varepsilon$. So if   $P(z)=\sum_{j=0}^d A_j z^j$ is self-inversive, then $A_j=\varepsilon \overline{A_{d-j}}$ for 
$j=0,\ldots, d$. 
In particular  $A_0=\varepsilon \overline{A_{d}}$ and $\overline{A_d}=\bar\varepsilon A_{0}$, so that $\varepsilon$ is necessarily of modulus $1$. 
 It is easy to check that if a polynomial  has all its zeros on the unit circle then it is self-inversive. In fact, Cohn \cite{C} proved  that a polynomial $P(z)$ has all its zeros on the unit circle if and only if it is self-inversive and its derivative $P'(z)$ has all its zeros in the closed unit disk $|z|\leq 1$. One might think that Cohn's result completely settles the matter. Indeed, Cohn's condition on $P'$ can be verified for a specific self-inversive polynomial, for instance by finding the zeros of $P'$, or checking that the Mahler measure of $P'$ is equal to the modulus of its leading coefficient. However, it may not be possible to use these methods for whole parametrized families of polynomials. We prove here an extension of Cohn's result -- see Theorem \ref{T-1} below --  which turns out to be more flexible than Cohn's Theorem for applications. We apply it to some polynomial families, including the polynomial family $P_k$ given by 
\begin{equation}\label{E-Pk}
P_k(z)=\frac{(2\pi)^{2k-1}}{(2k)!}\sum_{j=0}^k
(-1)^jB_{2j}B_{2k-2j}\binom{2k}{2j}z^{2j}+\zeta(2k-1)\left(z^{2k-1}+(-1)^kz\right).
\end{equation}
Here, as usual, the $B_{2j}$ are Bernoulli numbers. It was known from work of Lal\'\i n and Rogers \cite{LR} that the polynomials in this family had all their zeros on the unit circle for $k\leq 1000$. They  conjectured that this was true for all $k$. However, this conjecture had resisted previous attempts to prove it.  We do this in Theorem \ref{T-2} below, as an application of our main theorem (Theorem \ref{T-1}).

 There have also been a number of results in recent years that provide sufficient, easier-to-verify conditions for a self-inversive
polynomial to have its zeros on the unit circle. Lakatos \cite{L} proved that a reciprocal polynomial $\sum_{j=0}^{d}A_j z^j$ with real coefficients that satisfies 
\begin{equation}\label{Lakatos condition 1}
|A_d|\geq\sum_{j=0}^{d}\left| A_{j}-A_{d}\right|,
\end{equation}
has its zeros on the unit circle. Moreover, if the inequality is strict, the zeros are simple. 
Schinzel \cite{Sc2} improved Lakatos' result. His result -- see Corollary \ref{C-Schinzel} and Remark \ref{R-Lakatos} below -- follows from Theorem \ref{simple result}.

On the other hand, Lakatos and Losonczi \cite{LL2} proved that a reciprocal polynomial satisfying 
\begin{equation}\label{Lakatos condition 3}
|(1+\alpha)A_d|\geq\sum_{j=1}^{d-1}\left| A_{j}-(1-\alpha)A_{d}\right|
\end{equation}
for some $0\leq \alpha \leq 1$, has all its zeros on the unit circle. We have not been able to deduce their result using our approach, except in the case $\alpha=0$, see \eqref{Lakatos condition 1} and Remark \ref{R-Lakatos} below, 
and in the case $\alpha=1$,  which is  Corollary \ref{C-Lak}.
Their result was combined with Schinzel's into a more general statement,   proved in \cite{LL3}.

We will prove another result (Theorem \ref{simple result}) of the same general type that extends Schinzel's criterion in a different direction.

For other results concerning self-inversive and reciprocal polynomials, see
the books by Marden \cite[pp. 201--206]{Ma}, Rahman and  Schmeisser \cite{RS} and Schinzel \cite{Sc1}, as well as the papers by Ancochea \cite{A}, Bonsall and Marden \cite{BM1,BM2} and O'Hara and Rodriguez \cite{OR}.

In Section \ref{S-2} we state and prove our main result (Theorem \ref{T-1}), and deduce various consequences of it, including Theorem \ref{simple result} and Lemma \ref{L-2}. In Section \ref{S-3} we apply Lemma \ref{L-2} to prove that the polynomial family $P_k$ given by (\ref{E-Pk}) have all their zeros on the unit circle. Then in Section \ref{S-4} we give some details of how the same result, due originally to Lal\'\i n and Rogers \cite{LR}, can be proved in a similar way for two other polynomial families $Q_k$ and $W_k$.

\section{Results}\label{S-2}

Our main theorem is the following.

\begin{thm}\label{T-1} Let $h(z)$ be a nonzero complex polynomial of degree $n$ having
all its zeros in the closed unit disk $|z|\leq1$. Then for $d>n$ and any
$\lambda$ on the unit circle, the self-inversive polynomial 
\beq\label{E-Ph} 
P^{\{\lambda\}}(z)= z^{d-n}h(z)+\lambda h^*(z) 
\enq 
has all its zeros on the unit circle.

Conversely, given a self-inversive polynomial $P(z)$ having all its zeros on the unit circle, there is a polynomial $h$ having all its zeros in $|z|\leq1$ such that $P$ has a representation (\ref{E-Ph}). In particular, we can take $h(z)=\frac{1}{d}P'(z)$.
\end{thm}
Here $h^*(z)=z^n\overline{h}(1/z)$.
\begin{proof}

Assume first that the polynomial $h$, of degree $n$, has all its zeros in $|z|<1$.
 Then  $z^{d-n}h(z)$ has all of its zeros in the open unit disk while $h^*(z)$ has all of its zeros
with absolute value greater than 1. 
Now take $z$ such that $|z|=1$. We have 
\[ |h^*(z)|=|\overline{h}(1/z)|=|\overline{h(z)}|=|z^{d-n}h(z)|.
\] 
Assume for the time being that $\lambda$ has absolute value greater (respectively less) than $1$. Then
$| z^{d-n}h(z)|$ is less (respectively greater) than $|\lambda h^*(z)|$.
Hence, by Rouch\'e's Theorem, $P^{\{\lambda\}}(z)$ has all  of its zeros
in $|z|>1$  (respectively all of its zeros in $|z|<1$). As the zeros of $P^{\{\lambda\}}$ are
continuous functions of $\lambda$, we see that when  $|\lambda|=1$ then
$P^{\{\lambda\}}$ must have all its zeros on the unit circle.

The result under the weaker assumption that $h$ has all its zeros in the closed unit disc $|z|\leq1$ then follows by continuity.

Conversely, given $P$ self-inversive of degree $d$ with all its zeros on the unit circle, we note that  differentiating (\ref{E-self}) gives
\beq \label{E-self'}
P(z)=\frac{z}{d}P'(z)+\varepsilon \frac{z^{d-1}}{d}\overline{P'}(1/z),
\enq
which is of the form (\ref{E-Ph}) with $h(z)=\frac{1}{d}P'(z)$, $\la=\eps$ and $n=d-1$. Further,  the zeros of $P'$ certainly all lie in $|z|\leq1$. This is because  the zeros of $P'$ lie within the convex hull of the zeros of $P$, a result due originally to Gauss and Lucas --- see  \cite[pp. 23--24]{Ma} and \cite[pp. 72--73, pp. 92--93]{RS} for relevant references.
For completeness, and because of its elegance, we now reproduce a  proof,  by Ces\`aro, of this latter result, taken from \cite{RS}. 

Let $P(z)$ have zeros $\al_1,\ldots,\al_d$, and suppose that $P'(\be)=0$. If $\be$ equals some $\al_j$ then $\be$ is clearly in the convex hull of all the $\al_i$. So we can assume that $P(\be)\ne 0$, and then  on logarithmic differentation we have $P'(z)/P(z)=\sum_i\frac{1}{z-\al_i}$, and hence, putting $z=\be$ and taking complex conjugates, that $\sum_i\frac{1}{\overline{\be}-\overline{\al_i}}=0$.   Then 
 $\sum_i\frac{\be-\al_i}{| \be-\al_i|^2}=0$, giving $\be=\sum_i \la_i\al_i$, where $$\la_i=\frac{|\be-\al_i|^{-2}}{\sum_j |\be-\al_j|^{-2}}.$$
 Thus the $\la_i$ are all positive and sum to $1$.
 \end{proof}

It is not the case in general  that if (\ref{E-Ph}) holds for a particular $h$, and $P$ has all its zeros on the unit circle then $h$ must have all its zeros in $|z|\leq1$. For example, it is  known (see e.g., \cite[p. 9]{Sm}) that the polynomial $P(z)=z^k(z^3-z-1)+(z^3+z^2-1)$ has all its zeros on the unit circle for $k=0,1,\ldots,7$, while $h(z)=z^3-z-1$ has a zero  $1.3247 \ldots >1$.
However, in that direction we can say the following.

\begin{observation} For any $d>n$, let
 \[P_d^{\{\lambda\}}(z)= z^{d-n}h(z)+\lambda h^*(z).\]
If there is a $K>0$ such that for every $d>K$, $P_d^{\{\lambda\}}(z)$ has all its zeros on the unit circle, then $h(z)$ has all its zeros 
in the unit circle $|z|\leq 1$. 
\end{observation}
\begin{proof}
Assume that $h(z)$ has a zero $z_0$ with $|z_0|>1$. Take $\delta<|z_0|-1$, so that $|z_0|-\delta>1$. Then for $z$ on the circle $|z-z_0|=\delta$ and $d$
 sufficiently large we have
$$
|z^{d-n}h(z)|\geq (|z_0|-\delta)^{d-n}|h(z)|>|\la h^*(z)|.
$$ 
Hence, by Rouch\'e's Theorem, $P^{(\la)}(z)$ has the same number of zeros in the disc $|z-z_0|<\delta$ as $z^{d-n}h(z)$ has, namely at least one. This disc is completely
 outside the unit circle.
\end{proof}

As a simple consequence of Theorem \ref{T-1}, we obtain the following known result.

\begin{cor}[{{Lakatos and Losonczi \cite{LL1}}}]           \label{C-Lak}
 A self-inversive polynomial $P(z)=\sum_{j=0}^{d}A_j z^j$ satisfying 
\begin{equation}\label{E-Lak}
|A_d|\geq\tfrac12\sum_{j=1}^{d-1}|A_{j}|
\end{equation}
has all its zeros on the unit circle.
\end{cor}
\begin{proof}
We take
\[h(z)=\left\{\begin{array}{ll} A_dz^\frac{d}{2}+A_{d-1}z^{\frac{d}{2}-1} +\ldots +A_{\frac{d}{2}+1}z+\frac12 A_{\frac{d}{2}}
,& d \mbox{ even},\\\\
A_dz^\frac{d-1}{2}+A_{d-1}z^{\frac{d-3}{2}} +\ldots +A_{\frac{d+1}{2}},& d \mbox{ odd}.
\end{array}\right.\] 
Then, since by (\ref{E-Lak}) the leading coefficient of $h(z)$ is at least as big as the sum
of the moduli of the other coefficients, $h(z)$ has no zeros in $|z|>1$. 
Now $P(z)/A_d=z^{\lfloor\frac{d+1}{2}\rfloor}h(z)+\varepsilon h^*(z)$, where $\varepsilon=A_0/\overline{A_d}$. 
Therefore,
 by Theorem \ref{T-1}, $P(z)$ has all its zeros on the unit circle.
\end{proof}

We can also deduce the next result from Theorem \ref{T-1}.

\begin{thm}\label{simple result} A self-inversive polynomial $P(z)=\sum_{j=0}^{d}A_j z^j$ satisfying 
\begin{equation}\label{Smyth's condition}
|A_d|\geq\tfrac12\inf_{\substack{\mu\in\mathbb{C}\\|\mu|=1}}\sum_{j=0}^{d-1}\left|A_{j}-\mu A_{j+1}\right|
\end{equation}
 has all its zeros on the unit circle.
 \end{thm}
\begin{proof}
We first consider the case $\mu=1$, and look at $P(z)(z-1)$. We take \[h(z)= \left\{\begin{array}{ll} A_dz^\frac{d}{2}+ \left(A_{d-1}-A_d\right)z^{\frac{d}{2}-1}+\ldots &\\ 
+\left(A_\frac{d}{2}-A_{\frac{d}{2}+1}\right),& d \mbox{ even},\\\\
A_dz^\frac{d+1}{2}+ \left(A_{d-1}-A_d\right)z^{\frac{d-1}{2}}+\ldots&\\ 
+\left(A_\frac{d+1}{2}-A_{\frac{d+3}{2}}\right)z+\frac{1}{2}\left(A_{\frac{d-1}{2}}-A_{\frac{d+1}{2}}\right),& d \mbox{ odd}.
\end{array}\right.\] 
Again, the absolute value of the leading coefficient of $h(z)$ is at least as big as the sum
of the moduli of the other coefficients. This implies that $h(z)$ has no zeros in $|z|>1$ and we can apply the Theorem \ref{T-1} with $\lambda=-\varepsilon $ to conclude that $P(z)(z-1)$, and therefore $P(z)$,
 has all its zeros on the unit circle.
 
 To obtain the result in general, let $Q(z):=P(\mu z)$ and apply what has been proved to $Q(z)$, using the fact that $\left|\mu^jA_{j}-\mu^{j+1} A_{j+1}\right|=\left|A_{j}-\mu A_{j+1}\right|$.  
\end{proof}

\begin{cor}[{{Schinzel \cite{Sc2}}}]           \label{C-Schinzel}
A self-inversive polynomial $P(z)=\sum_{j=0}^{d}A_j z^j$ satisfying 
\begin{equation}\label{Lakatos condition 2}
|A_d|\geq\inf_{\substack{c,\mu\in\mathbb{C}\\|\mu|=1}}\sum_{j=0}^{d}\left|c
A_{j}-\mu^{d-j} A_{d}\right|,
\end{equation}
must have all of its zeros on the unit circle.
\end{cor}

\begin{proof}
Consider the self-inversive polynomial $P(z)=\sum_{j=0}^dA_jz^j$
 and assume that \eqref{Lakatos condition 2} is satisfied. Then we claim that (\ref{Smyth's condition}) holds.
We first  check this for $\mu=1$.  Indeed, by applying twice the triangle inequality,
\begin{align*}
\sum_{j=0}^{d-1} |A_j-A_{j+1}|&\leq  \sum_{j=0}^{d-1} |A_j-1/cA_d|+  \sum_{j=0}^{d-1} |1/cA_d-A_{j+1}|\\
&=2 \sum_{j=0}^{d} |A_j-1/cA_{d}|-2|1-1/c||A_d| \\
&\leq 2|1/cA_d| -2|1-1/c||A_d|\\
&\leq 2|A_d|.
\end{align*}

Then, again applying the result to $P(\mu z)$ for general $\mu$ on the unit circle gives the full result.
\end{proof} 

\begin{rem} \label{R-Lakatos} The condition \eqref{Lakatos condition 1} of Lakatos is the special case $c=\mu=1$ of \eqref{Lakatos condition 2}.
\end{rem}

We next show that the result of Theorem \ref{T-1} still holds if $P^{\{\lambda\}}(z)$ is perturbed by a small self-inversive `error' polynomial.

\begin{lem}\label{L-2} Let $h$ and $\lambda$ be as in Theorem \ref{T-1}, with
$|h(z)|\geq c>0$ for $|z|=1$. Let $e(z)$ be a polynomial of degree $m$ such that $|e(z)|\leq c$ for
$|z|=1$. Then  for $k>\max\{m,n\}$,   the self-inversive  polynomial \[
z^{2k-n}h(z)+z^{k}e(z)+\lambda(h^*(z)+z^{k-m}e^*(z)) \] has all its
zeros on the unit circle.
\end{lem}
\begin{proof}
We first assume that for some positive $c'<c$ we have $|e(z)|\leq c'<c$ for all $z$ with  $|z|=1$.
Now $h_e(z)=z^{k-n}h(z)+e(z)$ is a polynomial of degree $k$.
Because $|z^{k-n}h(z)|\geq c>c'\geq|e(z)|$ for $|z|=1$, Rouch\'e's
Theorem tells us that $h_e$ has all its zeros in the open unit disk
$|z|<1$. Also $h_e^*(z)=z^kh_e(1/z)=h^*(z)+z^{k-m}e^*(z)$.
 Now apply Theorem \ref{T-1} with $h$ replaced by $h_e$ and $d$ replaced by $k$.
 
 The general case, where we assume only that $|e(z)|\leq c$ for $|z|=1$, then follows by continuity.
 \end{proof}

 \section{Application to the polynomials $P_k$}\label{S-3}

 Let $k\geq 2$, and, as in \cite{LR}, define $P_k(z)$ by \eqref{E-Pk}.
 The study of this polynomial is motivated by the fact that it appears in a formula by Ramanujan \cite[~p. 276]{Be2}:
\begin{equation}\label{zeta(3) series}
\frac{1}{2z^k} P_k(z)=
(-z)^{-(k-1)}\sum_{n=1}^{\infty}\frac{1}{n^{2k-1}(e^{2\pi n z}-1)}-z^{k-1}\sum_{n=1}^{\infty}\frac{1}{n^{2k-1}(e^{2\pi n/z}-1)},
\end{equation}
valid for $z \not \in i \Q$. 
A variant of the polynomial $P_k(z)$ (without the term with the $\zeta$-value) was first considered by Gun, Murty, and Rath \cite{GMR} in the context of expressing
the special value of the $\zeta$-function as an Eichler integral that could yield information about the algebraic nature of the number. Murty, Smyth, and Wang \cite{MSW} studied this
variant of the polynomial  and found that all but four of its zeros lie on the unit circle. Finally, other variants of $P_k(z)$ 
were considered in \cite{LR}, and were shown to have all their zeros on the unit circle. However, the methods from \cite{LR} were not sufficient to prove that $P_k(z)$ itself
has all its zeros on the unit circle.

\begin{thm}\label{T-2} For all $k\in\N$, the polynomial $P_k$ has all its
zeros on the unit circle.
\end{thm}

 We need the following straightforward bounds.
 Put $q_j=\frac{\zeta(2j)\zeta(2k-2j)}{\zeta(2k)}$ for $j=0,1,\ldots, k$ and $\delta_j=q_j-\zeta(2j)$. Note that  $\delta_0=0$ and $\delta_j>0$ for $0<j<k$.
\begin{lem}\label{L-3}
\begin{itemize}
\item[(i)] For $n\geq 2$ we have $$1<\zeta(n)<1+\frac{n+1}{n-1}\cdot 2^{-n}.$$
\item[(ii)] For $k\geq 2$ and $j=1,2,\ldots,k-1$ we have
\[ 0<\frac{\zeta(2k-2j)}{\zeta(2k)}-1<3\cdot 4^{j-k}. \]
\item[(iii)] For $k\geq 11$  we have
\[ 0<\frac{\zeta(2k-1)}{\zeta(2k)}-1<\tfrac{11}{5}\cdot 4^{-k}. \]
\item[(iv)] For $k\geq 4$
and $j=1,2,\ldots, k-1$ we have
\[
|\delta_{j-1}-\delta_j|<\begin{cases} 21\cdot 4^{-k} \qquad\qquad\qquad\text{ if } j=1;\\
3\cdot 4^{-k}\left(4^j+\tfrac{2j-1}{2j-3}\right)\quad\text{ if } j\geq 2. \end{cases}
\]
\item[(v)] For $k\geq 2$
and $2\leq j\leq k/2$ we have $q_j=q_{k-j}$ and
\[ |q_{j-1}-q_j|<3\cdot
4^{-k}\left(4^j+\tfrac{2j-1}{2j-3}\right)+\tfrac{2j-1}{2j-3}\cdot 4^{1-j}.\]
\item[(vi)]  For $k\geq 4$ and $4\leq r\leq k$ we have

\[ \sum_{j=1}^r |\delta_{j-1}-\delta_j|<5 \cdot 4^{r-k}.
\]
 
\item[(vii)] For $k\geq 10$ and $r\geq 4$ we have
\[ \sum_{j=r+1}^{\lfloor\frac{k}{2}\rfloor}|q_{j-1}-q_j|< 5\cdot 2^{-k}+\tfrac{12}{7}\cdot 4^{-r}. \]
\end{itemize}
\end{lem}

\begin{proof}
 Part (i)-(iii)  are  easy -- see \cite[Lemmas 4.4 and 4.6]{MSW} for (i) and (ii). For (iv), we have, using (i) and (ii), that
 \begin{align*}
 \delta_j&=\zeta(2j)\left(\tfrac{\zeta(2k-2j)}{\zeta(2k)}-1\right)\\
 &<\left(1+\tfrac{2j+1}{2j-1}\cdot 4^{-j}\right)\cdot 3\cdot 4^{j-k}\\
 &=3\cdot 4^{-k}\left(4^j+\tfrac{2j+1}{2j-1}\right).
 \end{align*}
 Hence
 \[
 |\delta_0-\delta_1|=\delta_1<21\cdot 4^{-k},
 \]
 while in general 
 \[ |\delta_{j-1}-\delta_j|\le\max(\delta_{j-1},\delta_j),
 \]
 from which the result for $j\geq 2$ follows.

   For (v), we have 
    \[
|q_{j-1}-q_j|\leq |\delta_{j-1}-\delta_j|+\zeta(2j-2)-\zeta(2j),\\
\]
which gives the result using (ii) and (iv).

For (vi), we have, using (iv), that
\begin{align*}
\sum_{j=1}^r|\delta_{j-1}-\delta_j| &<21\cdot 4^{-k}+3\cdot 4^{-k}\sum_{j=2}^r \left(4^j+\tfrac{2j-1}{2j-3}\right)\\
&<4^{r-k}\left( 21\cdot 4^{-r}+3\sum_{j=-\infty}^r 4^{j-r}+3\cdot 3\cdot(r-1)\cdot 4^{-r}\right)\\
&<4^{r-k}\left(21/4^4+4+27/4^4\right)\\
&< 5\cdot 4^{r-k}.
 \end{align*}

 For (vii), as we have $j\geq 5$ in the summand, and $r\geq 4$, we obtain, using (v), that
 \begin{align*}
\sum_{j=r+1}^{\lfloor\frac{k}{2}\rfloor}|q_{j-1}-q_j|&\leq \sum_{j=r+1}^{\lfloor\frac{k}{2}\rfloor}\left(3\cdot 4^{-k}\left(4^j+\tfrac{2j-1}{2j-3}\right)+\tfrac{2j-1}{2j-3}\cdot 4^{1-j}\right)\\
& < \sum_{j=-\infty}^{\lfloor\frac{k}{2}\rfloor} 3\cdot 4^{j-k}+3\cdot 4^{-k}\cdot \tfrac{9}{7}\left(\left\lfloor\tfrac{k}{2}\right\rfloor-r\right)+\tfrac{9}{7}\sum_{j=r+1}^\infty 4^{1-j}\\
& \leq 3\cdot 4^{-k/2}\cdot \tfrac{4}{3}+2^{-k}\cdot \left(\tfrac{27}{7}\left(\tfrac{k}{2}-4\right)2^{-k}\right)+\tfrac97\cdot4^{-r}\cdot\tfrac43\\
& \leq 2^{-k}(4+1)+\tfrac{12}{7}\cdot 4^{-r}.
 \end{align*}

\end{proof}
We also need the standard identity \[
\frac{B_{2j}}{(2j)!}=(-1)^{j+1}\frac{2\zeta(2j)}{(2\pi)^{2j}}, \]
valid for all $j\geq 0$, since $B_0=1$ and $\zeta(0)=-1/2$. From this we see
that \[
P_k(z)=(-1)^k\frac{2}{\pi}\sum_{j=0}^k(-z^2)^j\zeta(2j)\zeta(2k-2j)
+ \zeta(2k-1)(z^{2k-1}+(-1)^kz). \]
Thus $P_k$ has leading coefficient $-\zeta(2k)/\pi$. To show that
$P_k$ has all its zeros on the unit circle it is sufficient to show
that the monic polynomial
$M_k(z)=-\frac{\pi}{\zeta(2k)}(z^2+1)P_k(z)$ has all its zeros on
the unit circle. (Most of the coefficients of $M_k$ are very small,
making it easier to work with than $-\frac{\pi}{\zeta(2k)}P_k$, whose coefficient of $z^{2j}$ is close to $2(-1)^{j+k+1}$ for most $j$.) We easily calculate that
\begin{align*}
M_k(z)=z^{2k+2}+(-1)^k-&\frac{\pi\zeta(2k-1)}{\zeta(2k)}(z^{2k+1}+z^{2k-1}
+(-1)^kz^3+(-1)^kz)\\ 
&+2\sum_{j=1}^k(-1)^jz^{2k+2-2j}(q_{j-1}-q_j).
\end{align*}
Hence, using the fact that for $k$ odd and $j=(k+1)/2$ we have $q_{j-1}-q_j=0$, we obtain
\begin{align*}
z^{-(k+1)}M_k(z)&=z^{k+1}+(-1)^kz^{-(k+1)}
-\frac{\pi\zeta(2k-1)}{\zeta(2k)}(z^{k}+z^{k-2}+(-1)^k(z^{2-k}+z^{-k}))\\
&\qquad+2\sum_{j=1}^{{\lfloor\frac{k}{2}\rfloor}}(q_{j-1}-q_j)(-1)^j(z^{k+1-2j}+(-1)^kz^{-(k+1-2j)})\\
&=z^{k+1}+(-1)^kz^{-(k+1)}-\pi(z^{k}+z^{k-2}+(-1)^k(z^{2-k}+z^{-k}))\\
&\qquad+2\sum_{j=1}^{r}(\zeta(2j-2)-\zeta(2j))(-1)^j(z^{k+1-2j}+(-1)^kz^{-(k+1-2j)})\\
&\qquad+2\sum_{j=r+1}^{{\lfloor\frac{k}{2}\rfloor}}(q_{j-1}-q_j)(-1)^j(z^{k+1-2j}+(-1)^kz^{-(k+1-2j)})\\
&\qquad+2\sum_{j=1}^r(\delta_{j-1}-\delta_j)(-1)^j(z^{k+1-2j}+(-1)^kz^{-(k+1-2j)})\\
&\qquad-\pi\left(\frac{\zeta(2k-1)}{\zeta(2k)}-1\right)(z^{k}+z^{k-2}+(-1)^k(z^{2-k}+z^{-k}))\\
&=z^{-(k+1)}H_r(z)+z^{-(k+1)}E_r(z),
\end{align*}
where 
\begin{align*}
H_r(z) &=z^{2k+2}+(-1)^k-\pi(z^{2k+1}+z^{2k-1}+(-1)^k(z^3+z))\\
&\qquad+2\sum_{j=1}^{r}(\zeta(2j-2)-\zeta(2j))(-1)^j(z^{2k+2-2j}+(-1)^kz^{2j})\\
E_r(z)&=-\pi\left(\frac{\zeta(2k-1)}{\zeta(2k)}-1\right)\left(z^{2k+1}+z^{2k-1}+(-1)^k(z^{3}+z)\right)\\
& \qquad+2\sum_{j=1}^r(\delta_{j-1}-\delta_j)(-1)^j(z^{2k+2-2j}+(-1)^kz^{2j})\\
& \qquad + 2\sum_{j=r+1}^{{\lfloor\frac{k}{2}\rfloor}}(q_{j-1}-q_j)(-1)^j(z^{2k+2-2j}+(-1)^kz^{2j}).
\end{align*}

Here $H_r$ is the main polynomial, with $E_r$ the error polynomial.

We can rewrite $H_r(z)$ as
\[ H_r(z)=z^{2k+2-2r}h_r(z)+(-1)^kh_r^*(z),
\]
where
\begin{equation}\label{E-h_r}
 h_r(z)= z^{2r}-\pi z^{2r-1}-\pi z^{2r-3}+2\sum_{j=1}^r(\zeta(2j-2)-\zeta(2j))(-1)^jz^{2r-2j}, 
 \end{equation}
and $E_r(z)$ as
\[ E_r(z)=z^{k+2}e_r(z)+(-1)^kze^*_r(z),\]
where
\begin{align*} e_r(z) &= -\pi\left(\frac{\zeta(2k-1)}{\zeta(2k)}-1\right)( z^{k-1}+ z^{k-3})
+2\sum_{j=1}^r(\delta_{j-1}-\delta_j)(-1)^jz^{k-2j}\\
&\qquad\qquad+2\sum_{j=r+1}^{{\lfloor\frac{k}{2}\rfloor}}(q_{j-1}-q_j)(-1)^jz^{k-2j}.
\end{align*}

We now take $r=4$. We then have the following bound.
\begin{lem}\label{L-6} For $k\geq 11$  and $|z|=1$ we have $|e_4(z)|\leq 0.019$.
\end{lem}
\begin{proof}  Take $z$ with $|z|=1$. Applying Lemma \ref{L-3} (iii), (vi) and (vii) 
  we have
\begin{align*}
|e_r(z)| & \leq \left|-\pi\left(\frac{\zeta(2k-1)}{\zeta(2k)}-1\right)( z^{k-1}+ z^{k-3})\right|
+\left|2\sum_{j=1}^r(\delta_{j-1}-\delta_j)(-1)^jz^{k-2j}\right|\\
&\qquad+\left|2\sum_{j=r+1}^{{\lfloor\frac{k}{2}\rfloor}}(q_{j-1}-q_j)(-1)^jz^{k-2j}\right|\\
& < 2\pi\cdot \tfrac{11}{5}\cdot 4^{-k}+ 2\cdot 5 \cdot 4^{r-k}+2\cdot 5\cdot 2^{-k}+2\cdot \tfrac{12}{7}\cdot 4^{-r},
\end{align*}
which is less than $0.019$ for $r=4$ and $k\geq 11$.
\end{proof}

\begin{proof}[Proof of Theorem \ref{T-2}] The result is known to be true for $k\leq 10$ -- see \cite{LR}. The polynomial
\begin{align*} h_4(z)&= z^8-\pi z^7+\left(1+\tfrac{\pi^2}{3}\right)z^6-\pi z^5+\left(\tfrac{\pi^2}{3}-\tfrac{\pi^4}{45}\right)z^4\\&+\left(\tfrac{2\pi^6}{945}-\tfrac{\pi^4}{45}\right)z^2+\left(\tfrac{2\pi^6}{945}-\tfrac{\pi^8}{4725}\right)\\
\end{align*}
given by (\ref{E-h_r})  has all its zeros of modulus less than $1$.  Furthermore, it is a matter of routine calculation to find  that its minimum on the unit circle occurs at 
 $z_0,\overline{z_0}\approx e^{\pm 0.20325951i}$, where  
$|h_4(z_0)|= 0.0214\ldots>0.020.$

So for $k\geq 11$ we can apply Lemma \ref{L-2} with $h=h_4$, $e=e_4$ and $c=0.020$ to obtain the required result. 
\end{proof}

\section{Applications to other polynomials}\label{S-4}

The method described in this work can be also used to study the other families of polynomials that appear in the statement of Claim 1.1 of \cite{LR}.
 First notice that the coefficients can be thought of as special values of variations of the Riemann zeta function or 
a Dirichlet $L$-function.
\begin{align}
(-1)^k\frac{\pi}{2^{2k+1}}Q_k(z):=&\sum_{j=1}^{k-1} \eta_0(2j)\eta_0(2k-2j) (-z^2)^j \notag\\& +(-1)^k\frac{\pi}{4} \eta_0(2k-1)(z^{2k-1}+(-1)^kz);\\
\frac{(-1)^k}{4}Y_{k}(z)=&\sum_{j=1}^{k-1} \eta_0(2j)\eta_0(2k-2j) z^j; \\
(-1)^k \frac{\pi}{2^{2k+1}}W_k(z)=&\sum_{j=0}^k  \eta(2j)\eta(2k-2j) (-z^2)^j;\\
\frac{(-1)^k}{(2k)!4} \left(\frac{\pi}{2}\right)^{2k+2}S_k(z)=&\sum_{j=0}^{k}L(2j+1,\chi_4)L(2k-2j+1,\chi_4) z^j.
\end{align}
We have used that
\[L(2j+1,\chi_4)=(-1)^j\frac{E_{2j}}{2(2j)!}\left(\frac{\pi}{2}\right)^{2j+1},\]
where the $E_{2j}$ are the Euler numbers given by 
\[\frac{2}{e^t+e^{-t}}=\sum_{n=0}^\infty \frac{E_n}{n!} t^n.\]
We have also used the notation
\[\eta(s)=(1-2^{1-s})\zeta(s);\]
and defined
\[\eta_0(s):=(1-2^{-s})\zeta(s).\]

The polynomial families $Y_k(z)$ and $S_k(z)$ were studied with the aid of Schinzel's result \eqref{Lakatos condition 2} in \cite{LR} and do not need 
further consideration here, by virtue of the fact that Corollary \ref{C-Schinzel} follows from Theorem \ref{simple result}. We proceed to outline the proofs for the other polynomials $Q_k$ and $W_k$.
It is easy to prove equivalent results to those of Lemma \ref{L-3}  for the other functions. Here we give the corresponding bounds for $\eta$ and $\eta_0$ (the functions involved in $Q_k(z)$ and $W_k(z)$), without proof. 

For $Q_k(z)$, we let $q_j=\frac{\eta_0(2j)\eta_0(2k-2j)}{\eta_0(2k-1)}$ for $j=0,1,\ldots, k$ and $\delta_j=q_j-\eta_0(2j)$. On the other hand, for $W_k(z)$, we put $q_j=\frac{\eta(2j)\eta(2k-2j)}{\eta(2k)}$ for $j=0,1,\ldots, k$ and $\delta_j=\eta(2j)-q_j$. 
As before, $\delta_0=0$ and $\delta_j>0$ for $0<j<k$ for both $\eta$ and $\eta_0$.

\begin{lem}\label{L-5}
\begin{itemize}
\item[(i)] For $n\geq 2$, 
\[1<\eta_0(n)<1+2^{-n};\]
\[1-2^{1-n} <\eta(n)<1. \]
\item[(ii)] For $k\geq 2$ and $j=1,2,\ldots,k-1$ we have
\[0<\frac{\eta_0(2k-2j)}{\eta_0(2k-1)}-1<2^{-2k+2j};\]
\[0<1-\frac{\eta(2k-2j)}{\eta(2k)}<2^{1-2k+2j}.\]
\item[(iii)] For $j=1,2,\ldots,k-1$, we have, 
\[ 0<\delta_j<2^{1-2k+2j}, \]
for both $\eta_0$ and $\eta$.
\item[(iv)] For $k\geq 2$
and $j=1,2,\ldots,k-1$ we have $q_j=q_{k-j}$ and
\[ |q_{j-1}-q_j|< 2^{2-2j} + 2^{2-2k+2j}  \mbox{  for } \eta_0;\]
\[ |q_{j-1}-q_j|< 2^{4-2j} + 2^{2-2k+2j}  \mbox{  for } \eta.\]
\item[(v)] For $r\geq 1$, we have, 
\[ \sum_{j=1}^r|\delta_{j-1}-\delta_j|<\tfrac{2}{3} \cdot 4^{r+1-k},\]
for both $\eta_0$ and $\eta$.
\item[(vi)] 
\[ \sum_{j=r+1}^{\lfloor\frac{k}{2}\rfloor}|q_{j-1}-q_j|<\tfrac{4}{3} \cdot (2^{-2r}+2^{2-k})  \mbox{  for } \eta_0;\]
\[ \sum_{j=r+1}^{\lfloor\frac{k}{2}\rfloor}|q_{j-1}-q_j|<\tfrac{16}{3} \cdot (2^{-2r}+2^{-k})  \mbox{  for } \eta.\]

\end{itemize}
\end{lem}

\subsection{The polynomial $Q_k$}

To study $Q_k$, let us consider the monic polynomial 
\begin{align*}
N_k(z)&=\frac{1}{2^{2k-1}\eta_0(2k-1)}Q_k(z)(z^2+1)\\
&=z^{2k+1}+z^{2k-1}+(-1)^kz^3+(-1)^kz+\frac{4(-1)^k}{\pi} \sum_{j=1}^k(q_{j-1}-q_j)(-1)^{j-1}z^{2j}.
\end{align*}
We can then write a similar decomposition 
\[N_k(z)z^{-k-1}=z^{-k}H_r(z)+z^{-k}E_r(z).\]
The main term is given by
\[H_r(z)=z^{2k-2r}h_r(z)+(-1)^kh_r^*(z),\]
where
\[h_r(z)=z^{2r}+ z^{2r-2}+ \frac{4}{\pi} \sum_{j=1}^r(\eta_0(2j-2)-\eta_0(2j))(-1)^{j-1}z^{2r-2j+1}.\]
On the other hand, the error term is given by
\[E_r(z)=z^{k}e_r(z)+(-1)^kze_r^*(z),\]
where
\[e_r(z)=\frac{4}{\pi} \sum_{j=r+1}^{\lfloor\frac{k}{2}\rfloor}(q_{j-1}-q_j)(-1)^{j-1}z^{-2j+k+1}+\frac{4}{\pi} \sum_{j=1}^{r}(\delta_{j-1}-\delta_j)(-1)^{j-1}z^{-2j+k+1}.\]
From this, for $|z|=1$,
\[|e_r(z)|\leq \tfrac{16}{3\pi}(2^{-2r}+2^{2-k})+\tfrac{2}{3\pi} \cdot 4^{r+2-k} \leq 0.14\]
for $r=2, k\geq 8$.

We need to consider \[h_2(z)=z^4-\tfrac{\pi}{2}z^3+z^2+\left(\tfrac{\pi^3}{24}-\tfrac{\pi}{2}\right)z.\]
It is not hard to verify that all the zeros have absolute value strictly less than 1, and that, for $z$ on the unit circle, we find that
\[ |h_2(z)|\ge |h_2(1)|= \tfrac{\pi^3}{24} - \pi + 2 = 0.1503\ldots .\]
Thus Lemma \ref{L-2} can be applied, using $h_2$, $e_2$ and $c=0.15$.  This finishes the proof for $Q_k(z)$.

\subsection{The polynomial $W_k$}

To study $W_k$, let us consider the monic polynomial 
\begin{align*}
V_k(z) &= \frac{\pi}{2^{2k}\eta(2k)}W_k(z)(z^2+1)\\
&=z^{2k+2}+(-1)^k+2 (-1)^k\sum_{j=1}^{k}(q_{j-1}-q_j)(-1)^{j-1}z^{2j}.
\end{align*}
We can then write
\[V_k(z)z^{-k-1}=z^{-k-1}H_r(z)+z^{-k-1}E_r(z).\]
In this case the main term is given by
\[H_r(z)=z^{2k+2-2r}h_r(z)+(-1)^kh_r^*(z),\]
where
\[h_r(z)=z^{2r}+2 \sum_{j=1}^r(\eta(2j-2)-\eta(2j))(-1)^{j-1}z^{2r-2j}.\]
The error term is given by
\[E_r(z)=z^{k+1}e_r(z)+(-1)^kz^2e_r^*(z),\]
where
\[e_r(z)=2\sum_{j=r+1}^{\lfloor\frac{k}{2}\rfloor}(q_{j-1}-q_j)(-1)^{j-1}z^{-2j+k+1}-2 \sum_{j=1}^{r}(\delta_{j-1}-\delta_j)(-1)^{j-1}z^{-2j+k+1}.\]
From this, for $|z|=1$,
\[|e_r(z)|\leq \tfrac{32}{3}(2^{-2r}+2^{-k})+\tfrac{1}{3}\cdot 4^{r+2-k} \leq 0.5\]
for $r=3, k\geq 6$.

We thus need to consider the polynomial.
\[h_3(z)=z^6+\left(1-\tfrac{\pi^2}{6}\right)z^4+\left(\tfrac{7\pi^4}{360}-\tfrac{\pi^2}{6}\right)z^2+\left(\tfrac{7\pi^4}{360}-\tfrac{31\pi^6}{15120}\right), \]
 which has all its zeros in $|z|<1$. Furthermore, for $z$ on the unit circle, we find that

\[ |h_3(z)|\geq |h_3(1)|=|h_3(-1)|=-\tfrac{31\pi^6}{15120} + \tfrac{7\pi^4}{180} - \tfrac{\pi^2}{3} + 2 = 0.5271\ldots \]

 So, again, Lemma \ref{L-2} can be applied,  using $h_3$, $e_3$ and $c=0.52$.  This concludes the proof for $W_k(z)$.

\subsection*{Acknowledgement}
We thank Mathew Rogers for bringing the problem of the zeros of $P_k$ to our attention and for very helpful discussions.

\end{document}